\documentclass[11pt]{amsart}

\usepackage[a4paper,margin=1in]{geometry}
\usepackage{amssymb}
\usepackage{amsmath}
\usepackage{amsthm}

\usepackage[shortlabels]{enumitem}
\usepackage{url}
\usepackage[colorlinks=true,linkcolor=blue,citecolor=blue,urlcolor=blue]{hyperref}
\usepackage{cleveref}
\newtheorem{theorem}{Theorem}[section]
\newtheorem{example}[theorem]{Example}
\newtheorem{lemma}[theorem]{Lemma}
\newtheorem{proposition}[theorem]{Proposition}
\newtheorem{corollary}[theorem]{Corollary}
\theoremstyle{definition}
\newtheorem{definition}[theorem]{Definition}

\DeclareMathOperator{\lcm}{lcm}
\DeclareMathOperator{\Aut}{Aut}
\newcommand{\N}{\mathbb{N}}

\newcommand{\R}{\mathbb{R}}

\begin{document}

\title{Quotient order density of ordinary triangle groups}

\author{Darius William Young}
\address{Department of Mathematics, The University of Auckland, Auckland, New Zealand}

\begin{abstract}
We show that the set of orders of finite quotients of every given ordinary triangle group has natural density zero, and we provide an explicit upper bound on the growth of these sets. More precisely, if $m=\lcm(r,s,t)$, then for every $k<\frac{1}{\phi(2m)}$ the number of orders of finite quotients of $\Delta^+(r,s,t)$ that are at most $x$ is $O\left(x/(\log x)^k\right)$.
\end{abstract}

\maketitle

\section{Introduction}
\label{sec:Intro}

Triangle groups have long been studied for their wide variety of interesting quotients, as well as for their connection to highly symmetric surfaces, geodesics on surfaces, and orientably regular maps. An ordinary triangle group is simply a group with presentation
\[\Delta^+(r,s,t)=\langle x,y,z\;|\;x^r=y^s=z^t=xyz=1\rangle,\]
with $r,s,t\in\N$.

Let $\N_x$ denote the set of natural numbers less than or equal to the real number $x$. Let $\mathcal{Q}_{r,s,t}$ be the set of orders of the finite quotients of the ordinary triangle group $\Delta^+(r,s,t)$, and let $\mathcal{Q}_{r,s,t}(x)=\mathcal{Q}_{r,s,t}\cap[1,x]$. More generally, for a finitely generated group $\Gamma$, let $\chi_\Gamma$ be the set of orders of the finite quotients of $\Gamma$, and let $\chi_\Gamma(x)=\chi_\Gamma\cap [1,x]$. We call $d(\chi_\Gamma)$ the \textit{quotient order density} of $\Gamma$ (where $d$ denotes natural density; see \Cref{def:NatDen}), and we refer to the function $x\mapsto|\chi_\Gamma(x)|$ as the \textit{quotient order growth function} of $\Gamma$. Unless stated otherwise, $\phi$ will be Euler's totient function.

The question of determining the quotient order growth of triangle groups appears to have originated with Michael Larsen \cite{larsen_how_2001}, who was interested in the Hurwitz group $\Delta^+(2,3,7)$. Larsen used the classification of finite simple groups and the notion of quotient dimension (see \cite[Definition 0.2]{larsen_how_2001}) to prove that the quotient order density of $\Delta^+(r,s,t)$ is zero for the six hyperbolic cases where $\frac{11}{12} < \frac{1}{r}+\frac{1}{s}+\frac{1}{t} < 1$. Larsen's main theorem \cite[Theorem 0.3]{larsen_how_2001} also directly implies the existence of an infinite family of triangle groups all having quotient order density zero.

Independently (and some years later), May and Zimmerman \cite{may_symmetric_2018} showed that the set of {\em odd\/}  orders of finite quotients of $\Delta^+(3,3,t)$ has density 0 when $t = 5, 7, 9$, and $11$. This was taken further in unpublished notes by Tom Tucker, who recognised the importance of groups with large cyclic normal subgroups. Using a 1976 theorem of Bertram \cite{bertram_large_1976} on such groups, he proved that if any one of $r$, $s$, and $t$ is coprime to the other two, then $d(\mathcal{Q}_{r,s,t})=0$.

In particular, Tucker's application of Bertram's theorem established that $d(\mathcal{Q}_{2,3,7})=0$ without relying on the classification of finite simple groups, but resulted in a weaker asymptotic upper bound than Larsen's.

Tucker also conjectured that $d(\mathcal{Q}_{r,s,t})=0$ for all $r,s,t$. That the density should be zero is perhaps surprising, given that triangle groups are known to have a wide variety of finite quotients. Every hyperbolic triangle group $\Delta^+ = \Delta^+(r,s,t)$ is residually finite 
(by a theorem of Mal'cev \cite{malcev_isomorphic_1940}), and also `SQ-universal', 
meaning that every countable group is isomorphic to a subgroup of some quotient of $\Delta^+$; see \cite{neumann_sq-universality_1973}.

In investigating Tucker's conjecture, we considered how and why his approach worked, and it soon became apparent that a modification of Bertram's theorem might help extend Tucker's proof to cover all triples $(r,s,t)$, but this required an extension of the asymptotic number theory used in \cite{bertram_large_1976}. This eventually led to the main theorem of this paper:

\begin{theorem}
\label{thm:DensityMain}
Let $r,s,t\in\N$ be a triple of positive integers and let $m=\lcm(r,s,t)$. Then for any $k<\frac{1}{\phi(2m)}$ there exists a constant $C>0$ such that
\[
|\mathcal{Q}_{r,s,t}(x)| \le \frac{Cx}{(\log x)^k}.
\]
In particular $d(\mathcal{Q}_{r,s,t})=0$.
\end{theorem}

\Cref{thm:DensityMain} has the following two immediate consequences for the full triangle groups and orientably regular maps (see \Cref{sec:TriangleGroups} and \Cref{sec:RegularMaps}):

\begin{corollary}
\label{cor:main2}
For every full triangle group $\Delta(r,s,t)$ we have $d(\chi_{\Delta(r,s,t)})=0$.
\end{corollary}

\begin{corollary}
\label{cor:regularMaps}
If $\mathcal{G}_{s,t}$ is the set of genera of the orientably regular maps of type $\{m,k\}$, then $d(\mathcal{G}_{s,t})=0$.
\end{corollary}

\begin{proof}
Let $M$ be an orientably regular map of type $\{m,k\}$ and genus $g$, with orientation-preserving automorphism group $G$. By the Riemann-Hurwitz formula $2g-2=|G|\left(\frac{1}{2}-\frac{1}{m}-\frac{1}{k}\right)$. If $\frac{1}{m}+\frac{1}{k}\ge\frac{1}{2}$, then $\frac{1}{2}-\frac{1}{m}-\frac{1}{k}\le 0$, so $2g-2\le 0$. Hence $g\le 1$, and so $\mathcal{G}_{m,k}\subseteq\{0,1\}$, which trivially has density zero.

Suppose that $\frac{1}{m}+\frac{1}{k}<\frac{1}{2}$, and let $c=\frac{1}{2}\left(\frac{1}{2}-\frac{1}{m}-\frac{1}{k}\right)$, so that $c>0$ and $g=1+c|G|$. The orientation-preserving automorphism group of any orientably regular map of type $\{m,k\}$ is a quotient of $\Delta^+(2,m,k)$ by a torsion-free normal subgroup. Let $\mathcal{A}_{m,k}$ be the set of orders of the orientation-preserving automorphism groups of orientably regular maps of type $\{m,k\}$. Then, since torsion-free normal subgroups are a subset of all normal subgroups, $\mathcal{A}_{m,k}\subseteq\mathcal{Q}_{2,m,k}$, and hence $d(\mathcal{A}_{m,k})\le d(\mathcal{Q}_{2,m,k})=0$ by \Cref{thm:DensityMain}. Since $g=1+c|G|$, every genus $g\le x$ in $\mathcal{G}_{m,k}$ comes from an order $|G|\le(x-1)/c$ in $\mathcal{A}_{m,k}$, so
\[
|\mathcal{G}_{m,k}\cap[1,x]|\le|\mathcal{A}_{m,k}\cap[1,(x-1)/c]|.
\]
Since $d(\mathcal{A}_{m,k})=0$ and $c$ is a fixed positive constant, the right-hand side divided by $x$ tends to $0$ as $x\to\infty$, and hence $d(\mathcal{G}_{m,k})=0$.
\end{proof}

\Cref{thm:DensityMain} is trivial for the spherical case, in which the groups are finite, and is straightforward for the Euclidean case. The Euclidean triangle groups are $\Delta^+(2,3,6)$, $\Delta^+(2,4,4)$, and $\Delta^+(3,3,3)$. By observations of Coxeter and Moser \cite[\S\S~8.3--8.4]{coxeter_generators_1984}, the smooth finite quotients of these groups have orders of the form $6(b^2 + bc + c^2)$, $4(b^2+c^2)$, and $3(b^2 + bc + c^2)$ respectively. A theorem of Bernays (see \cite{blomer_estimates_2006}) states that the set of integers represented by a positive definite binary quadratic form has density zero, with counting function $O(x/\sqrt{\log x})$, so the smooth quotients of the Euclidean triangle groups contribute a set of orders of density zero. Moreover, every non-smooth finite quotient of $\Gamma=\Delta^+(r,s,t)$ is a smooth quotient of some $\Delta^+(r',s',t')$ where $r'$ divides $r$, $s'$ divides $s$, and $t'$ divides $t$, with at least one a proper divisor. When $\Gamma$ is Euclidean, every such $\Delta^+(r',s',t')$ is spherical and therefore finite. Since there are only finitely many possibilities for $(r',s',t')$, $\Gamma$ has only finitely many non-smooth finite quotients, so these contribute only finitely many orders.

The hyperbolic case is more difficult, and accordingly, unless stated otherwise, every ordinary triangle group from now on in this paper can be assumed to be hyperbolic, although our methods can be applied to any infinite triangle group.

We begin in \Cref{sec:Background} by recalling some background on natural density and triangle groups. In \Cref{prop:threeprops}, we will define a ``large" subset $S$ of $\mathbb{N}$ dependent on $r,s,t$ that has trivial intersection with $\mathcal{Q}_{r,s,t}$. In \Cref{sec:EratosthenesSieve}, we apply \Cref{thm:Eratosthenes} to obtain a partial result about the growth of $|S\cap[1,x]|$. Then in \Cref{sec:ProofMainThm}, we combine our results with two lemmas of Bertram to deduce an upper bound for $|(\N\setminus S)\cap[1,x]|$ and hence for $|\mathcal{Q}_{r,s,t}(x)|$, proving our main theorem.

\section{Background}
\label{sec:Background}

\subsection{Natural density and asymptotic notation}

\begin{definition}
\label{def:NatDen}
    Let $A\subset\mathbb{N}$ be a subset of the natural numbers. If the limit
    \[
    l=\lim_{x\rightarrow\infty}\frac{|A\cap[1,x]|}{x}    
    \]
exists, then we say that the set $A$ has natural density $d(A)=l$.
\end{definition}

The natural density of a set gives a measure of the proportion of the natural numbers that lie in $A$.

\begin{example}
Take the set of prime numbers $\mathbb{P}$. By the prime number theorem $|\mathbb{P}\cap[1,x]|\sim \frac{x}{\log x}$, and hence there exists a constant $C>0$ such that

\[
d(\mathbb{P})=\lim_{x\to\infty}\frac{|\mathbb{P}\cap[1,x]|}{x}\le\lim_{x\to\infty} \frac{Cx}{x\log x}=\lim_{x\to\infty}\frac{C}{\log x}= 0.
\]

There are sets of numbers for which the limit does not converge and hence have no natural density. For example the set of natural numbers whose decimal expansion begins with $1$ has no natural density.

\end{example}

\subsection{Triangle groups}
\label{sec:TriangleGroups}
A \textit{full triangle group} $\Delta(r,s,t)$ is a group generated by three reflections over the edges of a non-degenerate geodesic triangle (which we shall call its \textit{fundamental triangle}), with angles $\frac{\pi}{r},\frac{\pi}{s},\frac{\pi}{t}$, for $r,s,t\in\N$ inscribed on either the sphere $S^2$, the Euclidean plane $\mathbb{E}^2$, or the hyperbolic plane $\mathbb{H}^2$. As such, we split triangle groups into three classes: spherical, Euclidean, and hyperbolic.

Triangle groups can be defined by a finite group presentation determined by the angles of the fundamental triangle. If $a,b,c$ represent reflections over the edges of a triangle with angles $\frac{\pi}{r},\frac{\pi}{s},\frac{\pi}{t}$, then the associated triangle group $\Delta(r,s,t)$ has presentation,
\[
\Delta(r,s,t)=\langle a,b,c\;|\;a^2=b^2=c^2=(ab)^r=(bc)^s=(ca)^t=1 \rangle.
\]
Since the angles of a triangle add up to more than, exactly, or less than $\pi$, depending on whether the triangle lies in the sphere, Euclidean plane, or hyperbolic plane, we can easily identify whether $\Delta(r,s,t)$ is spherical, Euclidean, or hyperbolic by observing whether $\frac{\pi}{r}+\frac{\pi}{s}+\frac{\pi}{t}$ is greater than, equal to, or less than $\pi$, or equivalently whether $\frac{1}{r}+\frac{1}{s}+\frac{1}{t}$ is greater than, equal to, or less than $1$. The spherical triangle groups are all finite, and include the symmetry groups of the Platonic solids.

Being generated by reflections, all full triangle groups contain both orientation-reversing and orientation-preserving elements. If we take the subgroup of orientation-preserving elements, we obtain the so-called \textit{ordinary triangle group} $\Delta^+(r,s,t)$, also known as a \emph{von Dyck group}, which is generated by the elements $x=ab$, $y=bc$, and $z=ca$. As such, this group has the presentation
\[\Delta^+(r,s,t)=\langle\, x,y,z\;|\;x^r=y^s=z^t=xyz=1\,\rangle.\]

Geometrically, the generators $x,y,z$ correspond to rotations by angles $\frac{2\pi}{r}$, $\frac{2\pi}{s}$, $\frac{2\pi}{t}$ around each corner of the fundamental triangle of $\Delta(r,s,t)$. Since $z^{-1}=xy$, the group $\Delta^+(r,s,t)$ can also be presented as $\langle x,y\,|\,x^r=y^s=(xy)^t=1 \rangle$ and is hence $2$-generated.

A \textit{smooth} quotient of an ordinary triangle group is simply any quotient in which the orders of the 
standard generators $x,y,z$ are preserved; as such, any smooth quotient $Q$ of an ordinary triangle group $\Gamma=\Delta^+(r,s,t)$ must be obtained as a quotient of $\Gamma$ by a torsion-free normal subgroup and have a presentation of the form
\[
Q=\langle x,y,z\;|\;x^r=y^s=z^t=xyz=\dots=1\rangle,
\]
with the property that $x$, $y$, and $z$ have orders $r$, $s$, and $t$.

\subsection{Regular maps}
\label{sec:RegularMaps}

Let $\Sigma_g$ be a genus $g$ closed connected surface without boundary, and let $X$ be a graph. Then a \textit{map} $M$ is an embedding of $X$ into $\Sigma_g$ such that all the faces of $X$ in $\Sigma_g$ are homeomorphic to open disks. In other words, $\Sigma_g\setminus X$ is a disjoint union of open disks.

A map is called \textit{orientable} or \textit{non-orientable} depending on whether the carrier surface $\Sigma_g$ is orientable or non-orientable. We also define the genus and Euler characteristic of a map $M$ to be the same as the genus and Euler characteristic of its carrier surface. We define the group $\Aut(M)$ of all automorphisms of $M$ as the group of all bijections from $M$ to $M$ that preserve the sets of vertices, edges, and faces of $M$ and incidences between them. If $M$ is orientable, then the set $\Aut^+(M)$ of orientation-preserving automorphisms of $M$ will be a subgroup of index $1$ or $2$. A map $M$ is called \textit{orientably regular} if it is orientable, and $\Aut^+(M)$ acts regularly on ordered edges.

As a consequence of the enforced high degree of symmetry, every vertex of a regular map must have the same valency and every face must have the same co-valency. A regular map $M$ is said to be of type $\{m,k\}$ if $m$ is the co-valency of every face and $k$ is the valency of every vertex. The orientation-preserving automorphism group of any such map has a presentation of the form
\[
\Aut^+(M)=\langle x,y,z\;|\;x^m=y^k=z^2=xyz=\dots=1\rangle,
\]
thus making $\Aut^+(M)$ a smooth quotient of the ordinary $(2,m,k)$ triangle group; see \cite{jones_dessins_2016} for more details.

Since $\Aut^+(M)$ acts regularly on ordered edges, an orientably regular map $M$ of type $\{m,k\}$ and genus $g$ with $G=\Aut^+(M)$ has $|G|/2$ edges, $|G|/m$ faces and $|G|/k$ vertices, and so Euler's formula gives a special case of the Riemann--Hurwitz formula
\begin{equation}
\label{eq:RH}
2g-2=|G|\left(\frac{1}{2}-\frac{1}{m}-\frac{1}{k}\right).
\end{equation}

\section{Quotient orders and restricted primes}
\label{sec:Restricted}

First, we give some elementary properties of quotients of ordinary triangle groups and establish the connection with groups having large cyclic normal Sylow subgroups.

\begin{proposition}
\label{prop:threeprops}
\leavevmode
\begin{enumerate}[{\rm (a)}]

\item Let $p$ be a prime and let $G$ be a finite quotient of $\Delta^+(r,s,t)$ of order $n$ with a non-trivial cyclic normal Sylow $p$-subgroup $P$, 
such that $C_G(P)$ has index $1$ or $2$ in $G$. Then $p$ divides at least one of $r$, $s$, and $t$.
\item Let $n$ be a positive integer having a prime factor $p$ such that
{\rm (i)} $p$ is coprime to $n/p$, and {\rm (ii)} the only divisor of $n$ congruent to $1$ mod $p$ is $1$. 
Then if $\gcd(p-1,q)\le 2$ for at least two values of $q\in\{r,s,t\}$ and each of $r$, $s$, and $t$ is coprime to $p$, then $n$ cannot be the order of a finite quotient of $\Delta^+(r,s,t)$.
\end{enumerate}
\end{proposition}

\begin{proof}
For part (a), the Schur-Zassenhaus Theorem \cite[Theorem 7.41]{rotman_introduction_1995} tells us that $G$ is a semi-direct product $P \rtimes H$ 
of the normal subgroup $P$ by some complementary subgroup $H$.  
If $C_G(P)$ has index $1$ in $G$, then $P$ is central in $G$, so $G$ is a direct product of $H$ and $P$, 
which implies that $P$ is a cyclic $p$-quotient of $G$ and hence of $\Delta^+(r,s,t)$, 
and so $p$ must divide at least two of $r$, $s$, and $t$.
On the other hand, if $C_G(P)$ has index $2$ in $G$, then $C_G(P)$ is a direct product of $C_H(P)$ and $P$, 
with $C_H(P)$ being characteristic in $C_G(P)$ and hence normal in $G$, so the quotient $G/C_H(P)$ is 
isomorphic to a dihedral group $P \rtimes C_2$, which implies that $p$ divides at least one of $r$, $s$, and $t$. 

For part (b), assume the contrary, and let $G$ be a finite quotient of $\Delta^+(r,s,t)$ of order $n$. 
Then by Sylow's theorems, conditions (i) and (ii) imply that $G$ has a cyclic normal Sylow $p$-subgroup $P$ of order $p$, 
and hence that $G$ is isomorphic to a semi-direct product $P \rtimes H$ for some subgroup $H$. 
Also $G/C_G(P)$ is isomorphic to a subgroup of $\Aut(P) \cong C_{p-1}$, and as $\gcd(p-1,q)\le 2$ for at least two values of $q\in \{r,s,t\}$ the images in $G$ of the 
three generators $x,y$, and $z$ of $\Delta^+(r,s,t)$ either centralise $P$ or invert each element of $P$ by conjugation, 
and therefore $C_G(P)$ has index $1$ or $2$ in $G$.  
But then part (a) implies that $p$ divides at least one of $r$, $s$, and $t$, a contradiction.
\end{proof}

So what we have established here is the surprisingly powerful result that if a positive integer $n$ has a prime factor satisfying a few number-theoretic conditions, then $n$ cannot be the order of an ordinary triangle group quotient. Most of the rest of the proof of our main theorem will be applying the methods of number theory to obtain an upper bound on the number of positive integers less than $x$ with no prime factors satisfying those properties.

We now recall a generalisation of a classical theorem of Mertens (see \cite{williams_mertens_1974}).

\begin{theorem}[Mertens' theorem for arithmetic progressions]
\label{thm:StrongMertens}
For any given pair of coprime positive integers $a$ and $b$, let $\mathcal{V}_x$ be the set of all primes $p\le x$ such that $p\equiv a$ mod $b$. Then there exists a positive constant $C(a,b)$ not dependent on $x$ such that
\[
\prod_{p\in\mathcal{V}_x}\left(1-\frac{1}{p}\right)\sim \frac{C(a,b)}{(\log x)^{1/\phi(b)}}.
\]
\end{theorem}

Next, we outline two key lemmas from \cite[pp.\,64--65]{bertram_large_1976}, 
with $g$ and $h$ being functions from $\R^+$ to $\R^{+}\!$.

\begin{lemma}
\label{lem:B1}
For all $x \in \R^+,$ the number of positive integers $n\le x$ such that $p^2$ divides $n$ for some prime $p>g(x)$ 
is less than $\frac{x}{g(x)}$.
\end{lemma}

\begin{lemma}
\label{lem:B2}
The number of positive integers $n\le x$ having a prime factor $p>h(x)$, and simultaneously having a divisor $d>1$ satisfying $d\equiv 1$ mod $p$, is less than $\frac{x((\log x) +1)}{h(x)}$.
\end{lemma}

These lemmas deal with two of the three conditions from \Cref{prop:threeprops}. The remaining condition is that $n$ must have a prime factor $p$ with $\gcd(p-1,q)\le 2$ for at least two values of $q\in\{r,s,t\}$. But as we shall see, one can easily loosen this to a better-studied condition on $p$.

\begin{definition}
For a fixed natural number $m$, we call a prime $p$ an $m$-restricted prime if $\gcd(p-1,m)\le 2$, and we let $\mathcal{P}^y_m(x)$ denote the set of $m$-restricted primes $p\in (y, x]$ for $x,y\in\R$. By slight abuse of notation, we denote the set of all $m$-restricted primes greater than $y$ by $\mathcal{P}_m^y(\infty)$.
\end{definition}

\begin{lemma}
\label{lem:Infntsize}
The set $\mathcal{P}_m^1(\infty)$ contains all of the prime terms of the arithmetic sequence $(2mn-1)_{n\ge 1}$, and in particular, contains infinitely many such primes.
\end{lemma}
\begin{proof}
Let $\mathcal{A}$ be the arithmetic sequence $(2mn-1)_{n\ge 1}$, or equivalently the sequence $(2mn+2m-1)_{n\ge 0}$. Clearly, $2m-1$ and $2m$ are coprime, so by Dirichlet's Theorem \cite[p. 371, (8.33)]{tenenbaum_introduction_2015} (or \Cref{thm:StrongMertens}) $\mathcal{A}$ contains infinitely many primes. If a prime $p$ is in $\mathcal{A}$ then $p=2mn-1$ for some $n$, hence $\gcd(p-1,m)\le 2$ and we are done.
\end{proof}

\section{An application of the Sieve of Eratosthenes}
\label{sec:EratosthenesSieve}

We recall the setup of the sieve method. Let $\mathcal{A}$ be a set of natural numbers less than $x\in \mathbb{R}$, with some set $\mathcal{P}$ of primes and subsets $\{\mathcal{A}_p\}_{p\in\mathcal{P}}$ of $\mathcal{A}$ indexed by the primes in $\mathcal{P}$. To each prime $p\in\mathcal{P}$, we associate a density $g(p)=\frac{\omega(p)}{p}$, with $\omega(p)\in\{0,\dots,p\}$ which we think of as roughly the proportion of $\mathcal{A}$ taken up by $\mathcal{A}_p$. Let $\mathcal{A}_1=\mathcal{A}$, and for any squarefree 
integers $d$ expressible as a product of distinct primes in $\mathcal{P}$, define
\[
\mathcal{A}_d:=\bigcap_{p|d}\mathcal{A}_p    
\]
and
\[
g(d):=\prod_{p|d}g(p).
\]

Let $z$ be a positive real number and let $P(z)=\prod_{\substack{p\in\mathcal{P}\\p<z}}p$. Let 
$S(\mathcal{A},\mathcal{P},z)$ denote the number of elements in the sifted set $\mathcal{A}\setminus 
\cup_{p|P(z)}\mathcal{A}_p$. We assume that there is some function $X:\R^+\to\R^+$ dependent only on $x$ such that
\[
|\mathcal{A}_d|=g(d)X+r_d
\]
for some remainder $r_d$.

Now recall the General Sieve of Eratosthenes \cite[Theorem 5.4.1]{cojocaru_introduction_2005}.

\begin{theorem}[The general sieve of Eratosthenes]
\label{thm:Eratosthenes}
Under the assumptions
\begin{enumerate}
    \item $|r_d| =O(d\cdot g(d))$,
    \item there exists a real number $\kappa\ge 0$, such that    
\[
\sum_{p|P(z)}g(p)\log(p)\le \kappa \log(z)+O(1),
\]

\item and there exists a real number $y>0$, such that $|\mathcal{A}_d|=0$ for every $d>y$,
\end{enumerate}
we have
\[
S(\mathcal{A},\mathcal{P},z)=XW(z)+O\left(\left(X+\frac{y}{\log(z)}\right)(\log(z))^{\kappa+1}\exp{\left(-\frac{\log y}{\log z}\right)}\right)
\]
where
\[
W(z):= \prod_{p|P(z)}\left( 1-g(p) \right).\]
\end{theorem}

\begin{proposition}
\label{prop:BetterSieveBound}
Let $m$ be a positive integer, let $\delta>1$, and let $f(x)=(\log x)^\delta$. If $S_x$ is the set of positive integers $n\le x$ with a prime factor in $\mathcal{P}_m^{f(x)}(\infty)$, then there exists a constant $M>0$ such that
\[
|\N_x\setminus S_x|\le Mx\left[\frac{(\log\log x)^{2}}{(\log x)}\right]^{\frac{1}{\phi(2m)}}.
\]
\end{proposition}

\begin{proof}

We wish to bound the size of the set $\N_x\setminus S_x$, and so we should start with the set of natural numbers less than $x$ and sieve out all numbers that are multiples of primes in $\mathcal{P}^{f(x)}_m(\infty)$ where $f(x)=(\log x)^\delta$ for some $\delta>1$. So for $z\le x$ we have $|\N_x\setminus S_x|=S(\mathcal{A},\mathcal{P}^{f(x)}_m(\infty),x)\le S(\mathcal{A},\mathcal{P}^{f(x)}_m(\infty),z)$. In the language of \Cref{thm:Eratosthenes} we take $\mathcal{A}=\N_x$ with $\mathcal{A}_p=\{n\le x\,:\,p|n\}$ and $P(z)=\prod_{p\in\mathcal{P}_m^{f(x)}(z)}p$. For a real number $l$, we define $\{l\} = l-\lfloor l\rfloor$. We can thus write $|\mathcal{A}_d|$ as
\[
|\mathcal{A}_d|= \frac{x}{d}-\left\{ \frac{x}{d} \right\}.
\]

So we should take $g(d)=\frac{1}{d}$ and $r_d=-\left\{ \frac{x}{d} \right\}$. Clearly $|r_d|<1 = d\cdot g(d)$, and also
\begin{align*}
\sum_{p\in\mathcal{P}^{f(x)}_m(z)}g(p)\log(p) &= \sum_{p\in\mathcal{P}^{f(x)}_m(z)}\frac{\log(p)}{p}\\
&\le \sum_{p<z}\frac{\log(p)}{p}\\
&\le \log(z)+O(1),
\end{align*}
where the final inequality is another theorem of Mertens \cite[Theorem 1.4.3]{cojocaru_introduction_2005}. So with $\kappa=1$, the second condition of \Cref{thm:Eratosthenes} is satisfied. Finally, we take $y=x$, and then all the conditions of \Cref{thm:Eratosthenes} are satisfied. Thus 
\[
S(\mathcal{A},\mathcal{P}^{f(x)}_m(\infty),z)=xW(z)+O\left(\left(x+\frac{x}{\log(z)}\right)(\log(z))^2\exp{\left(-\frac{\log x}{\log z}\right)}\right).
\]

Take $0<\epsilon<\frac{1}{3}$ and let $z$ be a function of $x$ defined by $\log z = \frac{\epsilon\log x}{\log \log x}$. We first examine the remainder term, which after plugging in $z$, and then simplifying comes out to $R=O\left(\frac{x}{(\log\log x)^2(\log x)^{v}}\right)$, where $v=\frac{1}{\epsilon}-2$. Turning to the main term, by \Cref{lem:Infntsize} and \Cref{thm:StrongMertens} we see that with $\mathcal{A}=(2mn-1)_{n\ge 1}$, we have
\begin{align*}
W(z)&=\prod_{p\in\mathcal{P}^{f(x)}_m(z)}\left(1-\frac{1}{p}\right)\\
&\le \prod_{\substack{p\in (f(x),z] \\ p \in\mathcal{A}}}\left(1-\frac{1}{p}\right)\\
&=\prod_{\substack{p\le z \\ p \in\mathcal{A}}}\left(1-\frac{1}{p}\right)\prod_{\substack{p\le f(x) \\ p \in\mathcal{A}}}\left(1-\frac{1}{p}\right)^{-1}.
\end{align*}

Next, observe that by \Cref{thm:StrongMertens} for any $\epsilon'>0$ there exists $x_0$ such that for all $x>x_0$ we have
\[
(1-\epsilon')C(\log z)^{-\frac{1}{\phi(2m)}}\le\prod_{\substack{p\le z \\ p \in\mathcal{A}}}\left(1-\frac{1}{p}\right)\le (1+\epsilon')C(\log z)^{-\frac{1}{\phi(2m)}},
\]
where $C$ is a constant dependent only on $m$.

Taking $\epsilon'$ to be sufficiently small and combining with the previous upper bound for $W(z)$, we see that with $C'=\frac{1+\epsilon'}{1-\epsilon
'}$ we have
\begin{align*}
    W(z)&\le\prod_{\substack{p\le z \\ p \in\mathcal{A}}}\left(1-\frac{1}{p}\right)\prod_{\substack{p\le f(x) \\ p \in\mathcal{A}}}\left(1-\frac{1}{p}\right)^{-1}\\
    &\le \left[\frac{(1+\epsilon')C}{(\log z)^{1/\phi(2m)}}\right]\cdot\left[ \frac{(\log f(x))^{1/\phi(2m)}}{(1-\epsilon')C}\right]\\
    &=C'\left(\frac{\log f(x) }{\log z}\right)^\frac{1}{\phi(2m)}.
\end{align*}

From here we may substitute for $\log z=\frac{\epsilon\log x}{\log \log x}$ to obtain
\[
W(z)= O\left[\left(\frac{\log f(x)\log\log x}{\log x}\right)^{\frac{1}{\phi(2m)}}\right].
\]

Since $f(x)=(\log x)^\delta$ for some $\delta>1$, the error term $R$ is asymptotically smaller than the main term $xW(z)$. Therefore, after picking an appropriate constant $M$, we obtain \Cref{prop:BetterSieveBound}.
\end{proof}

If we are not concerned with optimising the constants in the inequality in \Cref{prop:BetterSieveBound}, then we can simplify the formula as follows:

\begin{corollary}
\label{cor:BestBound}
Fix a positive integer $m$ and a real number $\delta>1$. Then for all $k<\frac{1}{\phi(2m)}$ there exists $C>0$ such that the number of positive integers less than $x$ with no $m$-restricted prime factor $p>(\log x)^\delta$ is less than
\[
\frac{Cx}{(\log x)^{k}}.
\]
\end{corollary}

\begin{proof}
    For any constant $\epsilon>0$, the function $\left[(\log \log x)^{2}/(\log x)^\epsilon\right]^{\frac{1}{\phi(2m)}}$ is positive and decreasing for all sufficiently large $x$, therefore it has a strict upper bound $c>0$. Let $M$ and $S_x$ be as in \Cref{prop:BetterSieveBound}, and let $k=\frac{1-\epsilon}{\phi(2m)}$, and let $C=Mc$. Then we can write
\begin{align*}
|\N_x\setminus S_x|\le Mx\left[\frac{(\log \log x)^2}{(\log x)}\right]^{\frac{1}{\phi(2m)}}&=Mx\left[\frac{1}{(\log x)^{1-\epsilon}}\right]^{\frac{1}{\phi(2m)}}\left[\frac{(\log \log x)^{2}}{(\log x)^\epsilon}\right]^{\frac{1}{\phi(2m)}}\\
&< Mcx\left[\frac{1}{(\log x)^{1-\epsilon}}\right]^{\frac{1}{\phi(2m)}}\\
&=\frac{Cx}{(\log x)^k}. \qedhere
\end{align*}
\end{proof}

\section{Proof of the main theorem}
\label{sec:ProofMainThm}

We now have all the components needed for our main theorem. All that remains to do is to combine them in the following lemma.

\begin{lemma}
\label{lem:MainLemma}
    Let $\delta>1$ be a real number and let $m$ be a positive integer. For all $k<\frac{1}{\phi(2m)}$ there exists a constant $C>0$ such that the number of positive integers $n\le x$ that do not have a prime factor $p$ with:
\begin{enumerate}[{\rm (a)}]
    \item $p>(\log x)^\delta$,
    \item if $d > 1$ is a divisor of $n$, then $d\not\equiv 1$ mod $p$,
    \item $\gcd(p,n/p) = 1$,
    \item $\gcd(p-1,m)\le 2$ (i.e. $p$ is $m$-restricted),
\end{enumerate}
is less than $\frac{Cx}{(\log x)^k}$.
\end{lemma}

\begin{proof}
Since the number of such integers $n\le x$ can only increase as $\delta$ increases, we may assume without loss of generality that $\delta\ge 2$.
Let $S_x$ be the set of natural numbers $n\le x$ with an $m$-restricted prime factor $p>(\log x)^\delta$. Let $H_x$ denote the set of natural numbers $n\le x$ with the property that if a prime $p>(\log x)^\delta$ is a factor of $n$, then $\gcd(p,n/p)=1$, and let $G_x$ denote the set of natural numbers $n\le x$ with the property that, if $p>(\log x)^\delta$ is a prime factor of $n$ then there is no factor $d>1$ of $n$ such that $d\equiv 1 \hbox{ mod }  p$.

Next, let $K_x=S_x\cap G_x\cap H_x$. Then every natural number $n\in K_x$ has a prime factor $p$ such that:

\begin{itemize}
    \item $p>(\log x)^\delta$,
    \item if $d> 1$ is any factor of $n$, then $d\not\equiv 1 \hbox{ mod }p$,
    \item $\gcd(p,n/p)=1$, and
    \item $\gcd(p-1,m)\le 2$.
\end{itemize}

By \Cref{cor:BestBound} and \Cref{lem:B1}, and \Cref{lem:B2} with $g(x)=h(x)=(\log x)^\delta$, we see that $|\N_x\setminus K_x|$ must be no larger than

\begin{equation*}
N(x) = \frac{x}{(\log x)^\delta} + \frac{x((\log x) +1)}{(\log x)^\delta} + \frac{C'x}{(\log x)^k}\le \frac{Cx}{(\log x)^k},
\end{equation*}
for a large enough constant $C$.
\end{proof}

We now have enough to prove \Cref{thm:DensityMain}.

\begin{proof}[Proof of \Cref{thm:DensityMain}]
By \Cref{prop:threeprops}, if $n$ is a positive integer with a prime factor $p$ such that
{\rm (i)} $p$ is coprime to $n/p$, and {\rm (ii)} the only positive integer divisor of $n$ congruent to $1$ mod $p$ is $1$, and {\rm (iii)} $\gcd(p-1,q)\le 2$ for at least two values of $q\in\{r,s,t\}$ and each of $r$, $s$, and $t$ is coprime to $p$, then $n$ cannot be the order of a finite quotient of $\Delta^+(r,s,t)$. So with $m=\lcm(r,s,t)$ and $(\log x)^\delta>m$ we see that by \Cref{lem:MainLemma}, for all $k<\frac{1}{\phi(2m)}$ there exists $C>0$ such that
\[
|\mathcal{Q}_{r,s,t}(x)|\le \frac{Cx}{(\log x)^k}. \qedhere
\]
\end{proof}

We would like to note that the constants in these bounds are not likely to be optimal. But for now, complete optimisation of the constants is not our primary concern; we are more interested in deriving their rough shape. Additionally, for many cases, this upper bound can be improved by alternative means (see \cite{larsen_how_2001}).

\section*{Acknowledgements}
The author is grateful to Marston Conder, Gabriel Verret and Jeroen Schillewaert, who have all played integral parts in helping to complete this project, and to the contributors on {\rm Math}{\em Overflow} who answered the questions \cite{young_density_2024,young_density_2024-1}. The author also acknowledges New Zealand's Marsden Fund for its financial support via Marston Conder's grant UOA2030.

\bibliographystyle{abbrv}
\bibliography{references}

\end{document}